\documentclass[10pt]{amsart}

\usepackage{amssymb,epsfig,ams math,amsthm}
\usepackage{stmaryrd}
\usepackage{enumerate}
\usepackage{cite}
\usepackage{comment}
\usepackage[all,cmtip]{xy}

\newtheorem{lemma}{Lemma}[section]
\newtheorem{prop}[lemma]{Proposition}
\newtheorem{thm}[lemma]{Theorem}
\newtheorem{thm'}[lemma]{``Theorem''}

\newtheorem{cor}[lemma]{Corollary}

\newtheorem*{uthm}{Theorem}

\theoremstyle{definition}
\newtheorem{defn}[lemma]{Definition}

\newtheorem{example}[lemma]{Example}

\newtheorem{Scon}[lemma]{Str{\o}m's construction}
\newtheorem{Ccon}[lemma]{Cole's construction}

\theoremstyle{remark}
\newtheorem{rmk}[lemma]{Remark}

\makeatletter
\let\c@equation\c@lemma
\makeatother
\numberwithin{equation}{section}

\newcommand{\cC}{\mathcal C}
\newcommand{\cD}{\mathcal D}
\newcommand{\cE}{\mathcal E}

\newcommand{\cI}{\mathcal I}
\newcommand{\cJ}{\mathcal J}

\newcommand{\cM}{\mathcal M}

\newcommand{\R}{\mathbb R}
\newcommand{\C}{\mathbb C}
\renewcommand{\L}{\mathbb L}
\newcommand{\F}{\mathbb F}
\newcommand{\M}{\mathbb M}

\newcommand{\ev}{\mathrm{ev}}
\newcommand{\id}{\mathrm{id}}
\newcommand{\dom}{\mathrm{dom}}
\newcommand{\cod}{\mathrm{cod}}
\newcommand{\bbL}{\mathbb L}
\newcommand{\bbR}{\mathbb R}
\newcommand{\Top}{\mathbf{Top}}
\newcommand{\Set}{\mathbf{Set}}

\newcommand{\sq}{\mathrm{Sq}}
\newcommand{\co}{\colon}

\newcommand{\po}[1][dr]{\save*!/#1+1.2pc/#1:(1,-1)@^{|-}\restore}
\newcommand{\pb}[1][dr]{\save*!/#1-1.2pc/#1:(-1,1)@^{|-}\restore}
\newcommand{\pbalt}[1][ur]{\save*!/#1-1.2pc/#1:(-1,+1)@^{|-}\restore}

\makeatletter\def\l@subsection{\@tocline{2}{0pt}{2pc}{5pc}{}}\makeatother

\title[On the construction of functorial factorizations]{On the construction of functorial factorizations for model categories}
\author{Tobias Barthel}
\address{Department of Mathematics, Harvard University,
Cambridge, MA \ 02138}
\email{tbarthel@math.harvard.edu}
\thanks{The first author was supported in part by a scholarship from Worcester College, Oxford, by the EPSRC, and by the Studienstiftung des deutschen Volkes.}
\author{Emily Riehl}
\address{Department of Mathematics, Harvard University,
Cambridge, MA \ 02138}
\email{eriehl@math.harvard.edu}
\thanks{The second author was supported in part by an NSF postdoctoral fellowship}
\date{\today}

\begin{document}
\begin{abstract}
We present general techniques for constructing functorial factorizations appropriate for model structures that are not known to be cofibrantly generated. Our methods use ``algebraic'' characterizations of fibrations to produce factorizations that have the desired lifting properties in a completely categorical fashion. We illustrate these methods in the case of categories enriched, tensored, and cotensored in spaces, proving the existence of Hurewicz-type model structures, thereby correcting an error in earlier attempts by others. Examples include the categories of (based) spaces, (based) $G$-spaces, and diagram spectra among others.
\end{abstract}

\maketitle

\section{Introduction}\label{intro}

In the late 1960s, Quillen introduced model categories, which axiomatize and thereby vastly generalize a number of classical constructions in algebraic topology and homological algebra. Somewhat ironically, a model category of spaces whose ``cofibrations'' were the classical, meaning Hurewicz, cofibrations and whose ``fibrations'' were the Hurewicz fibrations, established in \cite{stromhomotopy}, is somewhat difficult to obtain. The source of difficulties is two-fold. One has to do with subtleties involving point-set topology. The other obstacle is due to the fact that this model structure is not known to be \emph{cofibrantly generated}: while its fibrations are certainly defined by a lifting property, this lifting property is against a proper class of maps, and not simply a set. In the absence of this set-theoretical condition, there is no general procedure for constructing factorizations whose left and right factors satisfy the desired lifting properties.

In particular, while there exist natural notions of Hurewicz cofibrations and fibrations, Str{\o}m's ideas seem to be confined to the category of spaces.  Only in the last decade has there been progress toward Hurewicz-type model structures in one of their most natural settings: categories enriched, tensored, and cotensored over spaces  \cite{schwaenzlvogt, colemany}. Natural examples include based and unbased spaces, $G$-spaces, and diagram spectra. 
In the presence of a Quillen-type model structure, a Hurewicz-type model structure gives rise to a mixed model structure by an observation of Cole \cite{colemixing}. May and Ponto have advertized mixed model structures on topological spaces and categories of spectra \cite{maypontoconcise2} which combine Quillen- and Hurewicz-type model structures. The weak equivalences and fibrations are the weak homotopy equivalences and the Hurewicz fibrations; the cofibrant objects on spaces are the spaces of the homotopy types of CW complexes. May and Ponto argue that this is the model structure in which homotopy theory has always implicitly worked. For example, in the parametrized world \cite{maysig}, actual cell complexes are subtle and working in the mixed model structure promises a real simplification.

However, the difficulties inherent in this topic resurface in a mistake, recently noticed by Richard Williamson, in a crucial proof in \cite{colemany}, throwing the existence of Hurewicz-type model structures once more into doubt. The result claimed by Cole and proven here allows this philosophy to be applied to topological categories satisfying a smallness condition. In this paper, we present general techniques for producing factorizations for non-cofibrantly generated model categories that make use of the ``algebraic'' perspective on fibrations, explained below. We impose algebraic structures in order to replace point-set level arguments step-by-step with categorical ones, formulating a proof that is not specific to the category at hand. An interesting feature of this perspective is that it precisely identifies the flaw in Cole's proof and simultaneously suggests its solution.

A test case, spelled out in Section 3 and 4, illustrates how we might use the algebraic perspective to circumvent certain point-set level arguments in the construction of factorizations. Malraison and May \cite{malraisonfibrations, mayclassifying} observed that the Moore path space allows for an algebraic characterization of Hurewicz fibrations. Based on their results, we present a new factorization for the Str{\o}m model structure on topological spaces, which in particular avoids Str{\o}m's work on Hurewicz cofibrations \cite{stromnotes1, stromnotes2}. In fact, the construction of this factorization generalizes to any topologically bicomplete category, and we suspect that our arguments could also be used to establish the existence of Hurewicz-type model structures there. However, we prefer an alternative approach which can be more easily adapted to other (non-topological) contexts. This construction, outlined below, takes the ordinary path space as its point of departure but requires more elaborate algebraic machinery.

We explain our methods in analogy with the cofibrantly generated case. The starting point is an observation about Quillen's small object argument, due to Richard Garner \cite{garnercofibrantly,garnerunderstanding}. In a cofibrantly generated model category, a map is a fibration if and only if it has the right lifting property against a particular set of arrows; this is the case just when one can choose a solution to each such lifting problem. These chosen solutions are encoded as a solution to a single lifting problem involving the ``step-one'' factorization of Quillen's small object argument, which factors a map as a trivial cofibration followed by a map that is not typically a fibration. Put another way, a map is a fibration if and only if it admits the structure of an algebra for the (pointed) endofunctor that sends a map to its step-one right factor. In this way, Quillen's step-one factorization gives rise to an ``algebraic'' characterization of the fibrations in any cofibrantly generated model category.

By contrast, the Hurewicz fibrations in a topologically bicomplete category are not characterized by a lifting property against a set of maps. Nonetheless, we show that the ``step-one'' factorization produced by Cole, while not factoring a map into a trivial cofibration followed by a fibration, nonetheless provides a precise algebraic characterization of the fibrations. As above, the right factor of this factorization is a pointed endofunctor whose algebras are precisely the fibrations. A general categorical construction replaces this functorial factorization with another whose right functor is a monad whose algebras are again precisely the fibrations. In particular, the right factor is itself a free algebra and thus a fibration for entirely formal reasons: no point-set topology is necessary for this proof.

It remains to show that the left factor is a trivial cofibration for the model structure; again, on account of the algebraic perspective, no point-set topology is required. Instead, we use a composition criterion to show that the left functor so-constructed is a comonad. In particular, the left factor is a coalgebra for said comonad. For easy formal reasons, such coalgebras lift (canonically) against algebras for the right functor, which proves that the left factor is a trivial cofibration. This yields our main theorem.

\begin{uthm}
On any category $\cC$ enriched, tensored, and cotensored over spaces and satisfying a mild set-theoretical condition there exists a model structure whose fibrations, cofibrations, and weak equivalences are the $h$-fibrations, strong $h$-cofibrations, and homotopy equivalences respectively.
\end{uthm}

Abstractly our approach can be described as follows. Suppose given a category with two distinguished classes of morphisms: a class of ``fibrations'' that are characterized as algebras for a pointed endofunctor and a class of ``trivial cofibrations'' that are determined by a lifting property against the fibrations.  In practice, the former often arise as maps admitting solutions for a certain functorially constructed ``generic lifting problem''; the pointed endofunctor is obtained by pushing out this square. If the pointed endofunctor obtained in this manner satisfies a certain ``smallness'' condition, we can then construct a candidate functorial factorization by freely replacing it by a monad. If either of the following conditions hold \begin{itemize} \item[(a)] the category of algebras for the pointed endofunctor admits a vertical composition law, defined in Section \ref{moore2} below, or \item[(b)] the left factor in the factorization associated to the pointed endofunctor is a comonad \end{itemize} then these functors factor a map as a trivial cofibration followed by a fibration. By work of Garner \cite{garnerunderstanding}, both (a) and (b) hold automatically in the cofibrantly generated case, which is thereby subsumed. 

In fact, the methods of this paper generalize effortlessly to produce algebraic factorizations for any enriched bicomplete category equipped with an interval, i.e., a bipointed object, which satisfies a certain smallness condition. This can be used, for instance, to construct model structures on categories enriched in chain complexes. More details will appear in a forthcoming paper with Peter May.

Let us briefly compare this with previous work extending the small object argument to non-cofibrantly generated model categories. The main theorem of \cite{chorny} states that if \begin{itemize}\item [(i)] there is a cardinal $\kappa$ such that the domains of the arrows in the generating class of trivial cofibrations are $\kappa$-small and
\item[(ii)] there exists a functorial construction of a ``generic lifting problem'' in the sense of Remark \ref{redundantrmk}, 
 \end{itemize}
then an analogue of Quillen's small object argument can be used to construct appropriate functorial factorizations. In spaces, the Hurewicz fibrations are generated by the class $\{ A \to A \times I\}$ of cylinder inclusions. There are examples of spaces that are $\kappa$-small only if $\kappa$ exceeds the cardinality of their underlying set. Hence, for topological categories, the condition (i) is unreasonable. By contrast, our conditions (a) and (b) provide sufficient control over the left factor to allow us to weaken the smallness condition.

The structure of this paper parallels the gradual removal of point-set topology; in particular, we introduce categorical notions and results along the way as needed. In Section \ref{strommodel}, we review Str{\o}m's construction of a functorial factorization for the Hurewicz-type model structure on spaces and indicate why it is not suitable for generalization. In Sections \ref{moore1} and \ref{moore2}, we introduce the algebraic perspective on fibrations by considering the Moore path functorial factorization. In Section \ref{hmodel}, we discuss Hurewicz type model structures on any complete and cocomplete category that is tensored, cotensored, and enriched in topological spaces, and prove our main theorem. Finally, in the appendix, we explain the problem with Cole's factorization and present a few more details about our  construction.

\subsection*{Acknowledgments}

We would like to thank Richard Williamson for bringing this problem to our attention and Richard Garner, whose work inspired much of this paper. This work also benefitted from conversations with Richard Garner, Peter May, Bill Richter, Mike Shulman, and Richard Williamson. We would like to thank the referee for many helpful comments and several simplifications. 
In particular, the referee noticed that our original version of Proposition \ref{charfib} could be strengthened, which led to a considerable simplification in the proof of Theorem \ref{technicalmainthm}.
The first author would also like to thank Harvard University for its hospitality. 

\section{Str{\o}m's model structure on spaces}\label{strommodel}

In this short section, we review a few select details of Str{\o}m's construction of a model structure on the category of topological spaces and continuous maps whose weak equivalences are homotopy equivalences, fibrations are Hurewicz fibrations, and cofibrations are closed Hurewicz cofibrations. 

\begin{rmk}\label{convtop}
Even though Str{\o}m works in the category of all topological spaces, we restrict ourselves to a \emph{convenient} category of spaces, denoted $\Top$, which in particular should be cartesian closed. The two most prominent examples are $k$-spaces and compactly generated weak Hausdorff spaces. For a detailed discussion of these point-set issues, we refer the interested reader to  \cite[Ch.~1]{maysig}.
\end{rmk}

Write $I$ for the unit interval, topologized in the standard way with endpoints $0,1$. Recall that \emph{Hurewicz fibrations} are those maps in $\Top$ that have the homotopy lifting property, i.e., the right lifiting property with respect to all inclusions $A \xrightarrow{i_0} A \times I$. Dually, the \emph{Hurewicz cofibrations} are those maps with the homotopy extension property, i.e., the left lifting property against all projections $Y^I \xrightarrow{p_0} Y$.

There is a subtle, but important point here: In order to organize the aforementioned classes of maps into a model structure on $\Top$, we need the (model structure) cofibrations and trivial cofibrations, i.e., those cofibrations that are also homotopy equivalences, to be precisely the maps that lift against the (model structure) trivial fibrations and fibrations, respectively. In general, however, the Hurewicz cofibrations and fibrations don't have this property, but this can be fixed by requiring the (model structure) cofibrations to be \emph{closed} Hurewicz cofibrations.

\begin{rmk} 
Cofibrations in the category of compactly generated spaces are automatically closed, but interestingly this is not the case in the category of $k$-spaces; cf. \cite[1.6.4]{maysig}. Compare with the notion of \emph{strong cofibrations}, introduced by \cite{schwaenzlvogt}, which we discuss in Section \ref{hmodel}.
\end{rmk}

The factorization axiom (CM5) is the most difficult one to establish for this model structure, and here it suffices to construct the trivial cofibration -- fibration factorization. The main point-set level input to demonstrate this factorization is the following result of Str{\o}m's \cite[Thm.~3]{stromnotes1}:

\begin{prop}\label{stromchar}
If $i\co A \to X$ is an inclusion of a strong deformation retract such that there exists a map $q\co X \to I$ with $q^{-1}(0) =A$, then $i$ is a closed Hurewicz cofibration as well as a homotopy equivalence.
\end{prop}

This enables Str{\o}m to prove:

\begin{prop}\label{stromfact}
Every continuous map $f\co X \to Y$ can be factored as a homotopy equivalence and closed Hurewicz cofibration $i$ followed by a Hurewicz fibration $p$.  
\end{prop}

The proof can be found in \cite[Prop.~2]{stromhomotopy}, building on earlier work \cite{stromnotes2}. Here, we merely describe the construction so as to highlight the difficulties of na\"{i}vely extending it to topologically enriched categories. Str{\o}m's factorization makes use of 

\begin{defn}\label{defn:mappingspace}
The \emph{mapping path space} $Nf$ of $f\co X \to Y$ is defined to be the pullback 
\[
\xymatrix{Nf \ar[r]^-{\chi_f} \ar[d]_{\phi_f} \pb & Y^I \ar[d]^{p_0}  \\
X \ar[r]_f & Y &}
\]
\end{defn}

\begin{Scon}
Any map $f$ can be factored as $f = \pi \circ j$, with $j\co X \to Nf$ the map that sends a point of $X$ to the constant path at its image under $f$ and $\pi \co Nf \to Y$ evaluation of paths at their endpoint. This map $j$ is not necessarily a cofibration, so Str{\o}m factors it  through the space $E$ formed by gluing \[ E = X \times I \bigcup\limits_{X \times (0,1]} Nf \times (0,1]\] along $j$ and the inclusion of the half open interval. The map $j$ factors as $i \co x \mapsto (x,0)$ followed by the natural projection $\pi'$ obtained by including $E$ into $Nf \times I$ and projecting to the mapping space.  The result is a commutative diagram
\[
\xymatrix{X \ar[r]^f \ar[d]_{i} \ar[dr]^{j} & Y \\
E \ar[r]_{\pi'} & Nf \ar[u]_{\pi}
}
\]
Using Proposition \ref{stromchar} and \cite[Thm.~8 and 9]{stromnotes2}, Str{\o}m checks that $i$ is a trivial  cofibration and $p:=\pi \circ \pi'$ is a fibration. 
\end{Scon}

This construction generalizes without problems to any category $\cC$ enriched, tensored, and cotensored in spaces. In particular, we might define $E$ to be the pushout \[
\xymatrix{X \otimes (0,1] \ar[r]^-{X \otimes i'} \ar[d]_{j \otimes (0,1]} & X \otimes I \ar[d] \\
Nf \otimes (0,1] \ar[r]                                                      & E \po
}
\]
$i'$ being the map induced by the inclusion $(0,1] \to I$. However, Str{\o}m's characterization of trivial cofibrations is not available in the enriched context, so one needs to check directly that, among other things, the following lifting problem can be solved for any $A$:
\[
\xymatrix{A \ar[r] \ar[d]_{i_0}   & E \ar[d]^{p} \\
A \otimes I \ar[r] \ar@{-->}[ru] & Y
}
\]
But $E$ being defined as a colimit, it seems very difficult if not impossible to check that a lift exists and thus that $p$ is indeed a fibration. 

We will come back to this point in \S\ref{cole}.

\section{The Moore paths factorization I}\label{moore1}

We now present a second construction of the trivial cofibration -- fibration factorization for the $h$-model structure on $\Top$ in order to illustrate some of the key ideas involved in the ``algebraic'' perspective on homotopy theory. Following \cite{mayclassifying}, we introduce a functorial factorization based on the Moore path space to characterize the Hurewicz fibrations as \emph{algebras for a pointed endofunctor}. We use this characterization to prove that the right factor is a Hurewicz fibration and then apply Proposition \ref{stromchar} to show that the left factor is a closed Hurewicz cofibration and homotopy equivalence.

Interestingly, because this functorial factorization is particularly nice, the point-set topology input provided by Proposition \ref{stromchar} is not necessary to show that the left factor is a trivial cofibration. We will explain how this works in Section \ref{moore2}, introducing ideas that will be essential for our construction of a suitable functorial factorization for a general topologically bicomplete category in Section \ref{hmodel}.

\subsection{The Moore path space}

Let $Y$ be a space and let $\R^+ = [0, \infty)$. The space $\Pi Y$ of \emph{Moore paths} 
is defined to be the pullback
\begin{equation}\label{moorepath}
\xymatrix{ \Pi Y \pb \ar[d]_{\pi_{\mathrm{end}}} \ar[r] & Y^{\R^+} \times \R^+ \ar[d]^{\mathrm{shift}} \\ Y \ar[r]_-{\mathrm{const}} & Y^{\R^+}}
\end{equation}
The map ``shift'' is adjunct to the map given by precomposing with the addition map $\R^+ \times \R^+ \xrightarrow{+} \R^+$. It has the effect of reindexing a path so that it starts at the indicated time. Unpacking this definition,  $\Pi Y$ can be identified with the set of pairs $(p,t)$, where $t \in \R^+$ and $p\co [0,t] \to Y $ is a path in $Y$ of length $t$, topologized as a subspace of $Y^{\R^+}\times \R^+$.  The map $\pi_{\mathrm{end}}$ sends $(p,t)$ to $p(t)$.

Following \cite{mayclassifying}:

\begin{defn} The \emph{Moore path space} $\Gamma f$ of $f \co X \to Y$ is defined to  be the pullback
\begin{equation}\label{eq:Moorepullback}
\xymatrix{\Gamma f \ar[r] \ar[d]  \pb & \Pi Y \ar[d]^{\pi_0} \\
X \ar[r]_f                                                                                                     & Y 
}
\end{equation}
where $\pi_0\co (p,t) \mapsto p(0)$ is the evaluation of paths at $0$.  In other words, $\Gamma f$ is the set of triples $(p, t, x)$ where $(p,t)$ is a Moore path in $Y$ and $x$ is a point in the fiber over $p(0)$.
\end{defn}

As with $Nf$ above, we use the space $\Gamma f$ to define a factorization 
\begin{equation}\label{moorefact} \xymatrix{ X \ar[rr]^f \ar[dr]_{If} & & Y \\ & \Gamma f \ar[ur]_{Mf}}\end{equation}
The left factor $If \co X \to \Gamma f$ sends a point $x \in X$ to the length-zero path at $f(x)$. The right factor $Mf \co \Gamma f \to Y$ is the endpoint-evaluation map, obtained by composing the top map of (\ref{eq:Moorepullback}) with $\pi_{\mathrm{end}} \co \Pi Y \to Y$.

Unlike the case for the factorization constructed using the ordinary mapping space $Nf$, the Moore path space factors a map into a trivial cofibration $If$ followed by a fibration $Mf$. The proofs of these facts make use of the algebraic perspective on homotopy theory. To proceed, we need a few definitions.

\subsection{Functorial factorizations}
For the reader's convenience, we briefly review the notion of a functorial factorization. Let $\cC$ be any category. Denote by $\cC^{\mathbf{2}}$ the \emph{arrow category} of $\cC$; its objects are arrows of $\cC$, drawn vertically, and its morphisms are commutative squares which compose ``horizontally''. Write $\dom,\cod \co \cC^{\mathbf{2}} \rightrightarrows \cC$ for the evident forgetful functors, defined respectively by precomposing with the domain and codomain inclusions $\mathbf{1} \rightrightarrows \mathbf{2}$ of the terminal category into the category $\bullet \to \bullet$. 

\begin{defn}\label{defn:funfact}
A \emph{functorial factorization} consists of a pair of functors $L,R \co \cC^{\mathbf{2}} \to \cC^{\mathbf{2}}$ such that 
\[
\dom L = \dom, \qquad \cod R = \cod, \qquad \cod L = \dom R,
\]
and with the property that for any $f \in \cC^{\mathbf{2}}$, the composite (in $\cC$) of $Lf$ followed by $Rf$ is $f$.
\end{defn}
It is convenient to assign a name, say $E$, to the common functor $\cod L = \dom R \co \cC^{\mathbf{2}} \to \cC$ that sends an arrow to the object through which it factors. A functorial factorization factors a commutative square 
\begin{equation}\label{eq:functorialfactorization}
\vcenter{\xymatrix{ X \ar[r]^u \ar[d]_f & W \ar[d]^g \\ Y \ar[r]_v & Z}} \qquad \mathrm{as} \qquad 
\vcenter{\xymatrix{ X \ar[r]^u \ar[d]^{Lf} \ar@/_3ex/[dd]_f & W \ar[d]_{Lg} \ar@/^3ex/[dd]^g\\ Ef \ar[r]^{E(u,v)} \ar[d]^{Rf} & Eg \ar[d]_{Rg} \\ Y \ar[r]_v & Z}}
\end{equation}

The functors $L$ and $R$ are equipped with canonical natural transformations to and from the identity on $\cC^{\mathbf{2}}$ respectively, which we denote by $\vec{\epsilon} \co L \to \id$ and $\vec{\eta} \co \id \to R$. The components of these natural transformations at $f \in \cC^{\mathbf{2}}$ are the squares  
\[
\xymatrix{ X \ar[d]_{Lf} \ar@{=}[r] & X \ar[d]^f & & X \ar[d]_f \ar[r]^{Lf} & Ef \ar[d]^{Rf} \\ Ef \ar[r]_{Rf} & Y && Y \ar@{=}[r] & Y}
\]
In other words, $L$ and $R$ are \emph{pointed endofunctors} of $\cC^{\mathbf{2}}$, where we let context indicate in which direction the functors are pointed. An \emph{algebra} for the pointed endofunctor $R$ is defined analogously to the notion of an algebra for a monad, except of course there is no associativity condition in the absence of a multiplication map $\vec{\mu} \co R^2 \to R$. Similarly, a \emph{coalgebra} for the pointed endofunctor $L$ is defined analogously to the notation of a coalgebra for a comonad. Unpacking these definitions we observe:

\begin{lemma}\label{lem:Ralgebradefinition} $f \in \cC^{\mathbf{2}}$ is an $R$-algebra just when there exists a lift
\begin{equation}\label{eq:Ralgebradefinition}
\xymatrix{ X \ar@{=}[r] \ar[d]_{Lf} & X \ar[d]^f \\ Ef \ar[r]_{Rf} \ar@{-->}[ur]^t & Y}
\end{equation} 
Furthermore any choice of lift uniquely determines an $R$-algebra structure for $f$.
Dually, $i \in \cC^{\mathbf{2}}$ is an $L$-coalgebra just when there exists a lift 
\[
\xymatrix{ A \ar[d]_i \ar[r]^{Li} & Ef \ar[d]^{Ri} \\ B \ar@{-->}[ur]_s \ar@{=}[r] & B}
\]
Furthermore any choice of lift uniquely determines an $L$-coalgebra structure for $i$.
\end{lemma}

A key point, which we will make use of later, is expressed in the following lemma.

\begin{lemma}\label{lem:lifting} Any $L$-coalgebra $(i,s)$ lifts canonically against any $R$-algebra $(f,t)$.
\end{lemma}
\begin{proof}
Given a lifting problem, i.e., a commutative square $(u,v) \co i \to f$, the functorial factorization together with the coalgebra and algebra structures define a solution, namely the composite of the dashed arrows:
 \[\raisebox{6.75pc}{\xymatrix{ A \ar[r]^u \ar[d]_{Li} & X \ar@<.75ex>[d]^{Lf}  \\ Ei \ar@{-->}[r]^{E(u,v)} \ar@<-.75ex>[d]_{Ri} & Ef \ar[d]^{Rf} \ar@<.75ex>@{-->}[u]^t \\ B \ar@<-.75ex>@{-->}[u]_s \ar[r]_v & Y}}\qedhere \]
\end{proof}

\subsection{The Moore paths functorial factorization}

The construction (\ref{moorefact}) above defines a functorial factorization $I, M \co \Top^{\mathbf{2}} \to \Top^{\mathbf{2}}$ through the Moore path space. Furthermore, a classical result of May \cite[3.4]{mayclassifying} can be stated as follows:

\begin{prop}\label{maycharfib}
A map is a Hurewicz fibration if and only if it admits the structure of an $M$-algebra.
\end{prop}

Furthermore, as is noted in \cite{malraisonfibrations} and \cite{mayclassifying},  the pointed endofunctor $(M,\vec{\eta})$ extends to a monad $\M = (M, \vec{\eta}, \vec{\mu})$. This is the point at which Moore paths make their key contribution: composition of paths of variable lengths is strictly associative. In particular, the arrows $Mf$ are themselves (free) $M$-algebras, and are hence fibrations.

\begin{lemma}\label{lem:mooremonad}
The Moore paths functorial factorization extends to a monad $\M = (M, \vec{\eta}, \vec{\mu})$ over $\cod$ on the arrow category $\Top^{\mathbf{2}}$.
\end{lemma}
\begin{proof}
We need only define $\mu \co \Gamma Mf \to \Gamma f$, the domain component of the multiplication natural transformation $M^2 \to M$. A point in $\Gamma Mf$ is a Moore path $(p,t)$ in $Y$ together with a point in $\Gamma f$ --- this being itself a Moore path $(p',t')$ in $Y$ together with a point $x$ in the fiber of $p'(0)$ --- such that $p(0)= p'(t')$. The map $\mu$ sends this data to the concatenated path $pp'$ of length $t+t'$ together with the chosen point $x$ in the fiber over $pp'(0)=p'(0)$. The remaining details are left to the reader.
\end{proof}

These results allow for an easy proof of Proposition \ref{stromfact}.

\begin{cor}\label{cor:moorefactorization}
The factorization (\ref{moorefact}) factors $f$ into a trivial cofibration $If$ and a fibration $Mf$. 
\end{cor}
\begin{proof}
By Proposition \ref{maycharfib} and Lemma \ref{lem:mooremonad}, $Mf$ is a (free) $M$-algebra and hence a Hurewicz fibration, so the only thing to check is that $If$ is a trivial cofibration. But this follows immediately from Proposition \ref{stromchar}, using the map  $q\co \Gamma f \to [0,1]$ given by sending a Moore path $(p,t,x)$ of length $t$ to $\mathrm{min}(t,1)$. 
\end{proof}


An alternate proof that $If$ is a trivial cofibration, which avoids Str{\o}m's characterization \ref{stromchar}, was suggested by the referee. By Proposition \ref{maycharfib} and Lemma \ref{lem:lifting}, it suffices to show that $If$ is an $I$-coalgebra. Lemma \ref{lem:Ralgebradefinition} says that a map $i \co A \to B$ is an $I$-coalgebra if and only if there is a lift \[ \xymatrix{ A \ar[d]_i \ar[r]^{Ii} & \Gamma i \ar[d]^{Mi} \\ B \ar@{-->}[ur] \ar@{=}[r] & B}\]  This is the case, by the universal property of $\Gamma i$, if and only if $i$ extends to a \emph{Moore strong deformation retract}: a retraction $p$ of $i$ together with a Moore homotopy $h$ from $ip$ to $1_B$ whose components have length zero when restricted along $i$.
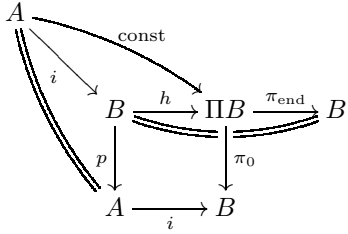
\[ \xymatrix{ A \ar[dr]_i \ar@/^/[drr]^{\mathrm{const}} \ar@{=}@/_/[ddr] \\ & B \ar[d]_p \ar[r]^-h \ar@{=}@/_2ex/[rr]|\hole & \Pi B \ar[d]^{\pi_0} \ar[r]^-{\pi_{\mathrm{end}}} & B \\ & A \ar[r]_i & B} \] 
Taking $i$ to be $If$, it is easy to check that the maps in the  pullback \eqref{eq:Moorepullback} define a Moore strong deformation retract, making $If$ an $I$-coalgebra and hence a trivial cofibration.

\begin{rmk} Note, in general the notions of $\M$-algebras (algebras for the full monad) and $M$-algebras (algebras for the pointed endofunctor) are distinct; the former is more restrictive. We will always take care to use a blackboard bold letter to distinguish algebras for the monad from algebras for the pointed endofunctor of the same name.
But in fact, because these functors arise in functorial factorizations, every $M$-algebra is a retract of an $\M$-algebra, namely, its right factor. In particular, a map has the left lifting property with respect to the $M$-algebras if and only if it has the left lifting property with respect to the $\M$-algebras. 
\end{rmk}

For the Moore paths factorization, $\M$-algebras are those Hurewicz fibrations that admit a ``transitive path lifting function'' in the terminology of \cite{mayclassifying}. The free algebras $Mf$ are both $\M$-algebras and $M$-algebras.

\section{The Moore paths factorization II}\label{moore2}

In fact, the Moore paths functorial factorization is an example of an \emph{algebraic weak factorization system}, defined below. This structure provides tighter algebraic control over the trivial cofibrations and fibrations, which in particular can be used to show that the left factor of a map always lifts against any algebra for the right factor.

Our proof that the Moore paths functorial factorization defines an algebraic weak factorization system uses a simple characterization, due to Richard Garner, that allows us to identify categories of algebras for the monad of an algebraic weak factorization system existing ``in the wild.''

 In Section \ref{hmodel}, by extending the methods introduced here, we will be able to construct functorial factorizations appropriate for categories that are enriched, tensored, and cotensored over topological spaces, but where point-set level characterizations of classes of maps in the ambient category are not generally available.

\subsection{Composition of algebras}

Let $L,R \co \cC^{\mathbf{2}} \to \cC^{\mathbf{2}}$ define a functorial factorization. To simplify the following discussion, we consider only algebras for the right factor $R$; dual results apply to the case of coalgebras for the left factor $L$.

\begin{defn}\label{defn:mapRalg}
A morphism $(u,v) \co f \to f'$ in $\cC^{\mathbf{2}}$, i.e., a commutative square (\ref{eq:functorialfactorization}) is a \emph{map of} $R$-\emph{algebras} if the square of lifts displayed in the interior of the cube 
\begin{equation}\label{eq:mapRalg}
\xymatrix@=15pt{ X \ar[dd]_{Lf} \ar[rr]^u \ar@{=}[dr] & & X' \ar@{=}[dr] \ar'[d]_{Lf'}[dd] \\ & X \ar[rr]^(.7)u \ar[dd]_(.65)f & & W \ar[dd]^{f'} \\ Ef \ar[dr]_{Rf} \ar'[r][rr]^(.4){E(u,v)} \ar@{-->}[ur]^s & & Ef' \ar@{-->}[ur]_{s'} \ar[dr]_{Rf'} \\ & Y \ar[rr]_v  & & Y'}
\end{equation}
commutes, i.e., if $u \cdot s = s' \cdot E(u,v)$.
\end{defn}

\begin{example}\label{ex:identityRalgebramap} The identity arrow at any object is always an $R$-algebra with a unique $R$-algebra structure given by its right factor.  Furthermore, for any $R$-algebra $(f,s)$, the map $(f,1_Y) \co (f,s) \to (1_X,R1_X)$ is an $R$-algebra map. The proof is a one line diagram chase: \[ R1_X \cdot E(f,1_Y) = 1_Y \cdot Rf = Rf = f \cdot s \] by (\ref{eq:functorialfactorization}) and (\ref{eq:Ralgebradefinition}).
\end{example}

Write $\mathbf{Alg}_R$ for the category of $R$-algebras and $R$-algebra maps. Via the forgetful functor $\mathbf{Alg}_R \to \cC^{\mathbf{2}}$, $R$-algebras can be viewed as objects in $\cC^2$. Using composition in $\cC$, objects and morphisms in  $\cC^{\mathbf{2}}$ can be composed ``vertically''. We say the category $\mathbf{Alg}_R$ \emph{admits a vertical composition law} if this composition operation can be lifted along the forgetful functor.

\begin{defn}\label{def:complaw} The category $\mathbf{Alg}_R$ \emph{admits a vertical composition law} if \begin{itemize} 
 \item[(i)] whenever $(f,s)$ and $(g,t)$ are $R$-algebras such that $\cod f = \dom g$, we can specify an $R$-algebra structure $t \bullet s$ for $gf$, in such a way that this composition operation is associative
 \item[(ii)] furthermore, for any maps $(u,v) \co (f,s) \to (f',s')$ and $(v,w) \co (g,t) \to (g',t')$ of $R$-algebras between composable pairs $(f,s), (g,t)$ and $(f',s')$, $(g',t')$, then $(u,w) \co (gf, t \bullet s) \to (g'f', t'\bullet s')$ is a map of $R$-algebras. 
\end{itemize}
\end{defn}

In other words, $\mathbf{Alg}_R$ admits a vertical composition law if both $R$-algebras and $R$-algebra maps can be composed vertically. This latter condition says that the vertical composite of the squares underlying $R$-algebra maps must again be an $R$-algebra maps with respect to the composite $R$-algebras. 

\begin{rmk} Concisely, a vertical composition law equips the category $\mathbf{Alg}_R$ with the structure of a \emph{double category}.
\end{rmk}

\begin{example}\label{ex:cgcomposition} For example, suppose $L,R$ are defined by pushing out from a particular coproduct of a set of generating trivial cofibrations $\cJ$ as in step one of Quillen's small object argument:
\begin{equation}\label{genlifting}
\xymatrix{ \cdot \ar[d]_{\coprod\limits_{j \in \cJ} \coprod\limits_{\mathrm{Sq}(j,f)} j} \ar[r]  \ar@{}[dr]|(.8){\ulcorner} & \cdot \ar[d]^{Lf} \ar@{=}[r] & \cdot \ar[d]^f \\ \cdot \ar[r] & \cdot \ar[r]_{Rf} & \cdot}
\end{equation}
By the universal property of the defining pushout, an $R$-algebra structure for $f$ is precisely a \emph{lifting function} $\Phi_f$, i.e., a choice of solution to all lifting problems against any $j \in \cJ$; see \cite[2.25]{riehlalgebraic}. 
\[ \vcenter{\xymatrix{ \cdot \ar[d]_{\coprod\limits_{j \in \cJ} \coprod\limits_{\mathrm{Sq}(j,f)} j} \ar[r]  \ar@{}[dr]|(.8){\ulcorner} & \cdot \ar[d]^{Lf} \ar@{=}[r] & \cdot \ar[d]^f \\ \cdot \ar[r] & \cdot \ar[r]_{Rf} \ar@{-->}[ur]_{\Phi_f} & \cdot}}\qquad\leftrightsquigarrow\qquad\vcenter{ \xymatrix{ \cdot \ar[d]_{\coprod\limits_{j \in \cJ} \coprod\limits_{\mathrm{Sq}(j,f)} j} \ar[r]  & \cdot \ar[d]^f \\ \cdot \ar[r] \ar@{-->}[ur]_{\Phi_f} & \cdot}} \]

Furthermore, a map of $R$-algebras $(f,s) \to (f',s')$ is precisely a commutative square from $f$ to $f'$ that respects the chosen lifts. 

The category $\mathbf{Alg}_R$ admits a vertical composition law.  The $R$-algebra structure assigned to the composite of $(f,\Phi_f), (g,\Phi_g) \in \mathbf{Alg}_R$ is the lifting function that solves \begin{equation}\label{eq:vertcomp} \xymatrix{ A \ar[dd]_j \ar[r]^a & X \ar[d]^f \\ & Y \ar[d]^g \\ B \ar[r]_b \ar@{..>}[ur] \ar@{-->}[uur] & Z}\end{equation} by first constructing the dotted lift according to $\Phi_g$, thereby obtaining a new lifting problem against $f$ whose dashed solution is chosen according to $\Phi_f$. It is easy to check that this composition law respects morphisms of $R$-algebras and is associative.
\end{example}


The category $\mathbf{Alg}_R$ of Example \ref{ex:cgcomposition} is isomorphic to a category of the following form. For a class of morphisms $\cJ \subset \cC^{\mathbf{2}}$, define a category $\cJ^\boxslash$ in which an object is an arrow $f$ of $\cC$ equipped with a lifting function $\Phi_f$ and whose morphisms $(u,v) \co (f, \Phi_f) \to (f',\Phi_{f'})$ are commutative squares so that the triangle of lifts displayed below commutes. \begin{equation}\label{eq:boxslashmap} \xymatrix{ A \ar[d]_j \ar[r]^a & X \ar[d]^(.65)f \ar[r]^u & X' \ar[d]^{f'} \\ B \ar@{-->}[ur] \ar[r]_b \ar@{-->}[urr] & Y \ar[r]_v & Y'}\end{equation} 
This definition can be extended to the case where $\cJ \hookrightarrow \cC^{\mathbf{2}}$ is a (typically non-full) sub\emph{category} of the arrow category. In this case, the lifts specified by a lifting function $\Phi_f$ must be natural with respect to morphisms $j' \to j \in \cJ$ in the sense that the following diagram of lifts commutes \[ \xymatrix{ A' \ar[r] \ar[d]_{j'} & A \ar[d]_(.3)j \ar[r] & X \ar[d]^f \\ B' \ar[r] \ar@{-->}[urr] & B \ar@{-->}[ur] \ar[r] & Y}\] 

Note there is a natural forgetful functor $\cJ^\boxslash \to \cC^{\mathbf{2}}$. The following proposition is easy to verify \cite[2.32]{riehlalgebraic}.

\begin{prop}\label{prop:boxslashcomp} The category $\cJ^\boxslash$ is equipped with a natural vertical composition law as displayed in \eqref{eq:vertcomp}.
\end{prop}

\begin{rmk}
For a generic functorial factorization, there is no reason for there to be a composition law for algebras of the right factor. However, we will see shortly that the existence of such a composition law is characteristic for functorial factorizations with good lifting properties. 
\end{rmk}

\subsection{Algebraic weak factorization systems}

The following definition is originally due to \cite{gtnatural}, with a small modification by Garner \cite{garnerunderstanding}.

\begin{defn}
An \emph{algebraic weak factorization system} on a category $\cC$ is pair $(\bbL, \bbR)$ with $\bbL = (L, \vec{\epsilon}, \vec{\delta})$ a comonad on $\cC^{\mathbf{2}}$ and $\bbR = (R, \vec{\eta}, \vec{\mu})$ a monad on $\cC^{\mathbf{2}}$ such that:
\begin{itemize}
 \item[(i)] $(L, \vec{\epsilon}), (R, \vec{\eta})$ give a functorial factorization on $\cC$, and
 \item[(ii)] The natural transformation $\Delta\co LR \to RL$ with components given by the commutative squares 
 \[
 \xymatrix{\cdot \ar[r]^{\delta_f} \ar[d]_{LRf} & \cdot \ar[d]^{RLf} \\
 \cdot \ar[r]_{\mu_f} & \cdot
 }
 \]
 is a distributive law, i.e., satisfies $\delta \circ \mu = \mu_L \circ E(\delta, \mu) \circ \delta_R$.
\end{itemize}
\end{defn}
It follows from (i) that $\cod R = \cod$ and that the codomain components of both $\vec{\mu}$ and $\vec{\eta}$ are the identity; dually, $\dom L = \dom$ and the domain components of $\vec{\delta}$ and $\vec{\epsilon}$ are identities. In other words, $\bbR$ is a monad over the functor $\cod$, and dually for $\bbL$.

\begin{defn}
The \emph{left class} of an algebraic weak factorization system $(\L,\R)$ is the class of maps that admit an $L$-coalgebra structure while the \emph{right class} is the class of maps that admit an $R$-algebra structure.
\end{defn}

Equivalently, the left class is the retract closure of the class of $\L$-coalgebras and the right class is the retract closure of the class of $\R$-algebras. Note by Lemma \ref{lem:lifting}, each map in the left class lifts against every map in the right class.

\begin{lemma}[Garner]\label{lem:recognition} If $(\L,\R)$ is an algebraic weak factorization system, then $\mathbf{Alg}_\R$ has a canonical vertical composition law.
\end{lemma}

A proof is given in \cite{garnerunderstanding}. We are particularly interested in the converse.

\begin{thm}[Garner]\label{thm:recognition}
If $\R$ is a monad on $\cC^{\mathbf{2}}$ over $\cod$ such that its category of algebras $\mathbf{Alg}_\R$ admits a vertical composition law, then there is a canonical algebraic weak factorization system $(\L,\R)$, with the functor $L \co \cC^{\mathbf{2}} \to \cC^{\mathbf{2}}$  defined by the unit. Furthermore,  the vertical composition law on $\mathbf{Alg}_\R$ determined by the algebraic weak factorization system $(\L,\R)$ coincides with the hypothesized one.
\end{thm}

Partial proofs can be found in \cite{garnerhomomorphisms,riehlalgebraic}, but  we felt that a more fleshed-out treatment was merited.

\begin{proof}
We make frequent use of the monadic adjunction $\cC^{\mathbf{2}} \rightleftarrows \mathbf{Alg}_\R$. The (non-trivial component of the comultiplication) $\vec{\delta}_f \co Lf \to L^2f$ is the domain component of the adjunct to the map \[ \xymatrix{ X \ar[dd]_f \ar[r]^{L^2f} & ELf \ar[d]^{RLf} \\ & Ef \ar[d]^{Rf} \\ Y \ar@{=}[r] & Y}\] Explicitly, $\delta$ is the composite of $E(L^2f,1)$ with the algebra structure assigned the composite of the free algebras $RLf$ and $Rf$. Because arbitrary maps $(u,v) \co f \to g$ give rise to maps $(E(u,v),v) \co Rf \to Rg$ of free $\R$-algebras, $\delta \co E \to EL$ is a natural transformation.

It remains to show that $\vec{\delta}$ gives $L$ the structure of a comonad in such a way that $(\L,\R)$ is an algebraic weak factorization system. We will check coassociativity and leave the unit and distributivity axioms to the reader.  

To this end, note that the following rectangles are maps of $\R$-algebras 
\[
\xymatrix{ Ef \ar[dd]_{Rf} \ar[r]^{\delta_f} & ELf \ar[d]_{RLf} \ar[r]^{E(1,\delta_f)} & EL^2f \ar[d]^{RL^2f} \\ & Ef \ar[r]^{\delta_f} \ar[d]_{Rf} & ELf \ar[d]^{Rf \cdot RLf} \\ Y \ar@{=}[r] & Y \ar@{=}[r] & Y} \qquad \xymatrix{ Ef \ar[dd]_{Rf} \ar[r]^-{\delta_f} & ELf \ar[r]^{\delta_{Lf}} \ar[d]_{RLf} & EL^2f \ar[d]^{RLf \cdot RL^2f} \\ & Ef \ar[d]_{Rf} \ar@{=}[r] & Ef \ar[d]^{Rf} \\ Y \ar@{=}[r] & \ar@{=}[r] & }
\]
We will show that the domain components agree by transposing both maps across the monadic adjunction. The domain component of the transpose of the left-hand map is 
 $E(1,\delta) \cdot L^2= L^3 = \delta_L \cdot L^2$, which is the domain component of the transpose of the right-hand map.  Hence $\delta$ is coassociative.

Finally, we verify that the vertical composition law arising from the algebraic weak factorization system by Lemma \ref{lem:recognition} agrees with the vertical composition we started with. The key observation is that for any composable pair of $\R$-algebras $(f,s)$ and $(g,t)$ we have the following map of $\R$-algebras:
\[
\xymatrix@C=40pt{ ELgf \ar[d]_{RLgf} \ar[r]^-{E(1,E(f,1))} & E(Lg \cdot f) \ar[d]_{R(Lg \cdot f)} \ar[r]^-{E(1,t)} & Ef \ar[r]^-s \ar[d]_{Rf} & X \ar[d]^f \\ Egf \ar[d]_{Rgf} \ar[r]^-{E(f,1)} & Eg \ar[d]_{Rg} \ar[r]^t & Y \ar@{=}[r] \ar[d]_g & Y \ar[d]^g \\ Z \ar@{=}[r] & Z \ar@{=}[r] & Z \ar@{=}[r] & Z}
\]
Recall $\delta_{gf}$ was defined to be $\mu_{gf} \bullet \mu_{Lgf} \cdot E(L^2{gf},1)$, where $\bullet$ is the given vertical composition law. By contrast, we write $\bullet'$ for the vertical composition given by the algebraic weak factorization system; by \cite[2.21]{riehlalgebraic}, $t \bullet' s$ is defined to be the composite \[ \xymatrix@C=40pt{ \cdot \ar[r]^{E(L^2gf,1)} & \cdot \ar[r]^{\mu_{gf} \bullet \mu_{Lgf}} & \cdot \ar[r]^{E(1,E(f,1))} & \cdot \ar[r]^{E(1,t)} & \cdot \ar[r]^s & \cdot}\] Because the above pasted rectangle is a map of $R$-algebras, the composite of the last four arrows is $t \bullet s \cdot E(s \cdot E(1,t) \cdot E(1,E(f,1)), 1)$. Precomposing with $E(L^2gf,1)$, we have a commutative diagram \[ \xymatrix@C=50pt{ \cdot \ar[d]_{E(L^2gf,1)} \ar[dr]|{E(L(Lg \cdot f),1)} \ar@/^/[drr]|(.6){E(Lf,1)} \ar@/^1.5pc/@{=}[drrr] \\ \cdot \ar[r]_{E(E(1, E(f,1)),1)} & \cdot \ar[r]_{E(E(1,t),1)} & \cdot \ar[r]_{E(s,1)} & \cdot \ar[r]_{t \bullet s} & \cdot}\] Hence $t\bullet' s = t \bullet s$.
\end{proof}

\subsection{The Moore paths algebraic weak factorization system}

We now use these results to show that the functorial factorization (\ref{moorefact}) is in fact an algebraic weak factorization system. This was noticed independently by Garner.

To this end, we must explain how define a vertical composition law for the category of $\M$-algebras. An $M$-algebra structure is classically called a \emph{path lifting function}. The function $\xi \co \Gamma f \to X$ specifying an $M$-algebra structure for $f \co X \to Y$ maps a Moore path $(p \co [0,t] \to Y, x \in X_{p(0)})$ to a point $\xi(p,t,x) \in X_{p(t)}$.  If $\xi$ is an $\M$-algebra structure, then this assignment must satisfy an additional ``transitivity'' condition; see Remark \ref{rmk:transitivity} below.

We might hope to use a procedure similar to the one outlined in Example \ref{ex:cgcomposition}. Suppose $g \co Y \to Z$, $\zeta \co \Gamma g \to Y$ is a second $M$-algebra. We can use $\zeta$ to lift the endpoint of a Moore path $(p \co [0,t] \to Z, x \in X_{p(0)})$ to $Y$, but we have lost too much information to proceed any further. 

The key idea is that an $M$-algebra structure determines a lift, displayed in the lemma below, that might be called a \emph{parametrized path lifting function}.

\begin{lemma}\label{lem:pathlifting}
There is an isomorphism, over $\cC^{\mathbf{2}}$,  between the category $\mathbf{Alg}_M$ and the category of arrows $f$ equipped with lifts
\begin{equation}\label{eq:pathliftingfunction}
\xymatrix{ \Gamma f \ar[d]_{i_0} \ar[rr] & & X \ar[d]^f \\ \Gamma f \times \R^+ \ar[r] \ar@{-->}[urr] & \Pi Y \times \R^+ \ar[r]_-{\ev} & Y}
\end{equation}
\end{lemma}
\begin{proof}
Clearly a parametrized path lifting function determines a path lifting function.  For the converse, first note that for any space $A$, the map $i_0 \co A \to A \times \R^+$ admits the structure of an $I$-coalgebra: The required lift $A \times \R^+ \to \Gamma i_0$ sends a point $(a,t) \in A \times \R^+$ to the path $r \mapsto (a,r)$ of length $t$ with fiber point $a$. Using this, we define the parametrized path lifting function to be the canonical lift of the $I$-coalgebra $i_0$ against the $M$-algebra $f$ obtained from the functorial factorization $(I,M)$, as in Lemma \ref{lem:lifting}. 

Explicitly, the diagonal arrow maps a pair consisting of a Moore path $(p \co [0,t] \to Y, x \in X_{p(0)})$ together with a parameter $s$ to the value of $\xi$ on the Moore path $(p \co [0,s] \to Y, x \in X_{p(0)})$.\footnote{If $s \le t$, this new $p$ is the restriction of the old one; if $s > t$, the new $p$ extends the old by remaining constant at $p(t)$ for the necessary duration.}
\end{proof}

\begin{rmk}\label{rmk:transitivity}
As detailed in \cite[3.2]{mayclassifying}, if $\xi$ is an $\M$-algebra, then the associated map (\ref{eq:pathliftingfunction}) is a \emph{transitive} parametrized path lifting function, which means that the lifted paths respect concatenation of paths in the following sense. If $p$ and $p'$ are composable paths of length $t$ and $t'$, and $x$ is in the fiber over $p(0)$, then the lift of the concatenated path agrees with the concatenation of the lift of the first path followed by the lift of the second path starting at $\xi(p,t,x)$. In this way, there is an isomorphism between $\mathbf{Alg}_\M$ and the category of arrows equipped with transitive parametrized path lifting functions.
\end{rmk}

\begin{prop} The category $\mathbf{Alg}_\M$ admits a vertical composition law.
\end{prop}
\begin{proof}
We explain how to compose the transitive path lifting functions associated to $\M$-algebras $(f,\xi)$ and $(g,\zeta)$ using the construction of Lemma \ref{lem:pathlifting}, i.e., we define a composite lift \[\xymatrix{ \Gamma (gf) \ar[dd]_{i_0} \ar[r]  &  X \ar[d]^f \\ & Y \ar[d]^g \\ \Gamma (gf) \times \R^+ \ar[r]_-{\mathrm{ev}} \ar@{-->}[uur]^{\zeta \bullet \xi} \ar@{..>}[ur] & Z}\]  The dotted lift sends a pair consisting of a Moore path $(p \co [0,t] \to Z, x \in X_{p(0)})$  and a parameter $s$ to the value of $\zeta$ on the Moore path $(p \co [0,s] \to Z, f(x) \in Y_{p(0)})$. This dotted map now allows us to define a new Moore path in $Y$: 
\[  r \mapsto \zeta(p,r,f(x)) \co [0,s] \to Y.\]  Call this path $\zeta(p)$. The point $x$ lies in the fiber over $\zeta(p)(0)$. Hence, $(\zeta(p), x) \in \Gamma f$. We define the dashed lift to be the map that sends our original Moore path $(p,t,x)$ and parameter $s$ to the point $\xi( \zeta(p), s, x)$. The remaining details are straightforward diagram chases, left to the reader.
\end{proof}

\begin{rmk}
One can wonder whether this proof applies to produce model structures in more general situations, i.e., for a category equipped with some kind of \emph{Moore path object}. An indication that this is indeed possible is given in \cite{garnertype}, who construct factorizations on so-called path categories. By weakening their axioms, Williamson \cite{williamsoncylindrical} obtains similar results in greater generality. 
\end{rmk}

\begin{rmk} 
Note that the only non-formal ingredient in the argument given here is the existence of a well-behaved Moore path object in topological spaces. Work in progress by Bill Richter is aimed at showing that the obvious analogue of \eqref{eq:Moorepullback} in a general category enriched, tensored, and cotensored over topological spaces has the same good properties. Our proof would then apply verbatim to yield a trivial cofibration -- fibration factorization for a Hurewicz-type model structure.

Our motivation for presenting a different, more abstract approach in the following sections stems mainly from its flexibility. Our algebraic methods apply to many contexts in which the fibrations are characterized by some generic lifting problem, but in which there exists no obvious analogue of Moore paths. For instance, this is the case for categories enriched, tensored, and cotensored over a category with an interval object. Examples include various model structures on dg-modules over a commutative differential graded algebra, as investigated by Peter May.
\end{rmk}

\section{Hurewicz model structures on topological categories}\label{hmodel}

Let $\cC$ be a \emph{topologically bicomplete} category, i.e., a bicomplete category enriched, tensored, and cotensored over some convenient category of spaces $\Top$. We will require one additional condition, akin to the ``smallness'' condition for Quillen's small object argument, which we will describe when we explain its purpose below. The tensor and cotensor structure suffices to abstract the definitions of homotopy equivalence, Hurewicz cofibration and Hurewicz fibration from Section \ref{strommodel}.

\subsection{Topological categories and Cole's construction}\label{cole}

In this section and the next we describe the heart of the construction in \cite{colemany}, set up the notation for the rest of the chapter and state some lemmata that will turn out to be useful in the proof of our main theorem.  \\

\begin{defn}
A \emph{homotopy} between two maps $f_0,f_1\co X \rightrightarrows Y$ in $\cC$ is a map $h\co X \otimes I \to Y$, or equivalently, a map $\widehat{h} \co X \to Y^I$ (its adjunct) such that 
\[ \vcenter{\xymatrix{ X \ar[d]_{i_0} \ar[dr]^{f_0} \\ X \otimes I \ar[r]^-h & Y \\ X \ar[u]^{i_1} \ar[ur]_{f_1}}} \qquad \mathrm{or~equivalently} \qquad \vcenter{\xymatrix{ & Y \\ X \ar[r]^-{\widehat{h}} \ar[ur]^{f_0} \ar[dr]_{f_1} & Y^I \ar[u]_{p_0} \ar[d]^{p_1} \\ & Y}} \] commutes, $i_0,i_1\co X \rightrightarrows X \otimes I$ and $p_0, p_1 \co Y^I \rightrightarrows Y$ being the morphisms induced by the two endpoint inclusions $* \rightrightarrows I$. 
\end{defn}

In particular, we have a notion of \emph{homotopy equivalence} in $\cC$.

\begin{defn}
A map $f$ in $\cC$ is an $h$-\emph{cofibration} if it has the left lifting property with respect to $p_0\co Z^I \to Z$ for all objects $Z \in \cC$. Dually, $f$ is an $h$-\emph{fibration} if it has the right lifting property with respect to all cylinder inclusions of the form $i_0\co Z \to Z \otimes I$. 
\end{defn}

Here the ``$h$'' stands for Hurewicz and also for homotopy. We would ideally like to construct a model structure on $\cC$ whose cofibrations are the $h$-cofibrations, whose fibrations are the $h$-fibrations, and whose weak equivalences are the homotopy equivalences. However, similarly to Section \ref{strommodel}, this is not possible because only some of the $h$-cofibrations lift against the class of $h$-fibrations that are also homotopy equivalences.

This motivates the following definition:

\begin{defn}
The class of \emph{strong cofibrations} is the class of maps that have the left lifting property with respect to the $h$-fibrations that are also homotopy equivalences.
\end{defn}

Because the maps $p_0 \co Z^I \to Z$ are homotopy equivalences and $h$-fibrations, cf. \cite{schwaenzlvogt}, strong cofibrations are in particular $h$-cofibrations. An immediate corollary of our main theorem, Theorem \ref{mainthm} below, establishes a so-called $h$-model structure, whose weak equivalences are homotopy equivalences, fibrations are $h$-fibrations, and cofibrations are the strong cofibrations. Henceforth, we use ``cofibrations'' and ``fibrations'' in the model structure sense, in particular dropping the ``$h$''.

It is possible to describe these right and left lifting classes using relative lifting properties \cite[4.2.2]{maysig}, 
but all we need is the following result. \\

\begin{lemma}\label{strongcof} $\quad$
 \begin{itemize}
  \item[(i)] The natural map $i_0\co A \to A \otimes I$ is a trivial cofibration for all objects $A \in \cC$.
  \item[(ii)] The class of (trivial) cofibrations is closed under retracts, pushouts and sequential colimits. 
 \end{itemize}
\end{lemma}
\begin{proof}
The proofs can be found in \cite{schwaenzlvogt} and \cite{maysig}. Part (ii) is immediate from the closure properties of any collection of arrows defined by a lifting property.
\end{proof}

\begin{Ccon}\label{colesconstruction}
Cole's construction attempts to factor an arbitrary map $f\co X \to Y$ in $\cC$ into a trivial cofibration followed by a fibration. 
To this end, start by forming the mapping path object $Nf$ of $f$, in precise analogy with Definition \ref{defn:mappingspace}. A new object $Ef$ is constructed by pushing out one of the projections from the pullback $\phi_f \co Nf \to X$ along  the natural map $i_0\co Nf \to Nf \otimes I$. Using the morphisms $f$ and $\widehat{\chi_f}$, the adjoint to the other projection $\chi_f$, we obtain an induced map $Rf\co Ef \to Y$ as shown in the following diagram.
\begin{equation}\label{eq:keydiagram}
\xymatrix{ Y^I \ar[r]^{p_0}            & Y  \\
Nf \ar[r]^-{\phi_f} \ar[d]_{i_0}   \ar[u]^{\chi_f}\pbalt                      & X \ar[u]^f \ar[d]_{Lf} \ar[rdd]^{f}\\
Nf \otimes I \ar[r]^-{\psi_f} \ar[rrd]_{\widehat{\chi_f}}  & Ef \ar[rd]|{Rf} \po \\
                                                                              & & Y
}
\end{equation}

In this way, we have factored $f$ as $Rf \circ Lf$ and furthermore, by Lemma \ref{strongcof}, the map $Lf\co X \to Ef$ is a trivial cofibration. If the map $Rf$ were a fibration, we would be done. However, this fails in general, so Cole proposes to iterate this construction, replacing $f$ by $Rf$, and applying the functorial factorization $(L,R)$ to the right factor. The eventual right factor of $f$ is defined by passing to the colimit $R^\omega f = \mathrm{colim}( Rf \to R^2f \to R^3f \to \cdots)$. The left factor of $f$ is then the composite $X \xrightarrow{Lf} Ef \xrightarrow{LRf} ERf \xrightarrow{LR^2f} ER^2f \to \cdots \to ER^\omega f$.

Because each map in the image of $L$ is a trivial cofibration, the left factor is a trivial cofibration. It
 remains to show that $R^{\omega}f$ is a fibration, which by \cite[5.2]{colemany} is equivalent to finding a lift in
\[
\xymatrix{N{R^{\omega}f} \ar[r]^{\phi_{R^{\omega}f}} \ar[d]_{i_0} & E{R^{\omega}f} \ar[d]^{R^{\omega}f} \\
N{R^{\omega}f} \otimes I \ar[r]_-{\widehat{\chi_f}} \ar@{-->}[ru] & Y
}
\]
To this end, \cite{colemany} asserts that the required lift is given by $\psi_{R^{\omega}f}$; however, the maps $\psi_{R^{n}f}$ do \emph{not} glue to induce a map $N{R^{\omega}f} \otimes I \to E{R^{\omega}f}$, cf. Section \ref{app1}. 
\end{Ccon}

We will see that there is a natural modification of the iterative part of Cole's construction that produces an algebraic weak factorization system with the appropriate homotopical properties.

\subsection{Algebraic characterization of fibrations}\label{sec:algchar}

The first key observation is that, even though the right factor $Rf$ fails to be an $h$-fibration, algebras for the pointed endofunctor $R$ are precisely $h$-fibrations.
The proof follows easily once we understand the universal property of the mapping space $Nf$. 

Fix a morphism $f\co X \to Y$ and let $\sq_f \co \cC^{\mathrm{op}} \to \Set$ be the functor that maps an object $A$ to the set of commutative squares of the form
\[
\xymatrix{A \ar[r] \ar[d]_{i_0} & X \ar[d]^f \\
A \otimes I \ar[r] & Y
}
\]
These squares correspond to lifting problems that test whether $f$ is an $h$-fibration.

\begin{lemma}\label{pathchar}
The functor $\sq_f$ is represented by the mapping path object $Nf$.
\end{lemma}
\begin{proof}
By the defining universal property of $Nf$, a map $\alpha: A \to Nf$ classifies a commutative square 
\begin{equation}\label{eq:representation}
\vcenter{\xymatrix{ A \ar[d]_{i_0} \ar[r]^u & X \ar[d]^f \\ A \otimes I \ar[r]_-v & Y}} \qquad \leftrightsquigarrow \qquad \vcenter{\xymatrix{ A \ar[dr]^{\alpha} \ar@/^1.5ex/[drr]^{\widehat{v}} \ar@/_1.5ex/[ddr]_u \\ & Nf \ar[r]^-{\chi_f} \ar[d]_(.4){\phi_f} \pb & Y^I \ar[d]^{p_0} \\ & X \ar[r]_f & Y}}
\end{equation}
\end{proof}

In particular, the identity map at $Nf$ classifies the right hand square in 
\begin{equation}\label{eq:liftingfactorization}
\xymatrix{ A \ar@/^2ex/[rr]^u \ar[d]_{i_0} \ar[r]_\alpha &Nf \ar[d]_{i_0} \ar[r]_-{\phi_f} & X \ar[d]^f \\ A \otimes I \ar[r]^-{\alpha \otimes I} \ar@/_2ex/[rr]_v & Nf \otimes I \ar[r]^-{\widehat{\chi_f}} & Y}
\end{equation}
which features prominently in the construction of the factorization (\ref{eq:keydiagram}). By the Yoneda lemma, or alternatively by adjointness, a square (\ref{eq:representation}) factors uniquely as the above diagram (\ref{eq:liftingfactorization}), where $\alpha \co A \to Nf$ is the classifying map. 

It is now easy to prove that the $h$-fibrations are precisely those objects in the image of the forgetful functor $\mathbf{Alg}_R \to \cC^{\mathbf{2}}$; in fact, 
any $h$-fibration lifts \emph{naturally} against the maps $i_0$:

\begin{prop}\label{charfib}
The category $\mathbf{Alg}_R$ is isomorphic to the category $\cI^\boxslash$ over $\cC^{\mathbf{2}}$, where $\cI$ is the 
 category whose objects are the maps $\{ i_0 \co A \to A \otimes I \mid A \in \cC \}$ and whose morphisms correspond to maps $A' \to A$ in $\cC$. 
\end{prop}
\begin{proof}
Suppose $(f,s) \in \mathbf{Alg}_{R}$. To solve a lifting problem \[\xymatrix{ A \ar[d]_{i_0} \ar[r]^u & X \ar[d]^f \\ A \otimes I \ar[r]_-v & Y}\] we first factor it as displayed in (\ref{eq:liftingfactorization}) and then factor the right hand square in (\ref{eq:liftingfactorization}) through the pushout of (\ref{eq:keydiagram}). This yields:
\[
\xymatrix{ A \ar[r]\ar[d]_{i_0} \ar@/^3ex/[rrr]^u & Nf \ar[d]_{i_0}  \ar[r]_-{\phi_f} & X \ar[d]^{Lf} \ar@{=}[r] & X \ar[d]^f \\ A \otimes I \ar[r] \ar@/_5ex/[rrr]_v & Nf \otimes I \ar[r]^-{\psi_f} \ar@/_2ex/[rr]_(.35){\widehat{\chi_f}}  & Ef \po \ar[r]^{Rf} \ar@{-->}[ur]^(.6)s & Y}\] The map $s$ defines an evident solution to the original lifting problem. 
Furthermore, the naturality of the isomorphism of Lemma \ref{pathchar} implies that the lifting functions defined in this manner are natural with respect to the morphisms in the category $\cI$.

Conversely, if $(f, \Phi_f) \in \cI^\boxslash$, then the solution, specified by $\Phi_f$, to the canonical lifting problem against 
the map $i_0 \colon Nf \to Nf \otimes I$ displayed in the left two squares above defines the map $s$ by the universal property of the pushout. The pair $(f,s)$ is then an $\R$-algebra. Uniqueness of the universal property of the pushout implies that these procedures define a bijection between the objects of the categories $\mathbf{Alg}_R$ and $\cI^{\boxslash}$.

The fact that maps in $\mathbf{Alg}_R$ agree with maps in $\cI^{\boxslash}$ follows easily from comparing the definitions \eqref{eq:mapRalg} and  \eqref{eq:boxslashmap} with the diagram above.
\end{proof}

\begin{rmk}\label{redundantrmk}
This argument shows that $f$ is an $h$-fibration if and only if there is a lift \[ \xymatrix{ Nf \ar[r]^{\phi_f} \ar[d]_{i_0} & X \ar[d]^f \\ Nf \otimes I  \ar@{-->}[ur] \ar[r]_-{\widehat{\chi_f}} & Y}\] 
as observed in \cite[5.2]{colemany}. We think of this square as presenting a ``generic lifting problem'' which detects fibrations. This is analogous to the lifting problem given by \eqref{genlifting} in the cofibrantly generated case.
\end{rmk}

At this point we are confronted with a problem: the algebras for the functor $R$ are precisely the fibrations, but because $R$ is not a monad, the maps $Rf$ are not themselves $R$-algebras. One idea is to try and replace the functor $R$ by its ``free monad'' $\F$, which is characterized by the property that the category of $\F$-algebras is isomorphic over $\cC^{\mathbf{2}}$ to the category of $R$-algebras (so in particular $\F$-algebras are precisely fibrations). There are two obstacles to implementing this idea. The first is set-theoretical. By an easy application of the monadicity theorem, the free monad $\F$ is equivalently specified by a left adjoint to the forgetful functor $\mathbf{Alg}_R \to \cC^{\mathbf{2}}$. However, it is not quite enough to simply know that an adjoint exists: the resulting monad on $\cC^{\mathbf{2}}$ might not be a monad over $\cod$ and thus not define the right factor in a functorial factorization. A theorem of Kelly, described in the next section, exhibits a certain smallness condition on $R$ under which the free monad  ``exists constructively''; in this case, a functorial factorization is produced. 

A second obstacle remains. Supposing that the free monad $\F$ exists constructively, it is not clear a priori that the left factor will still be a trivial cofibration because this construction involves quotienting. However, we can show that the factorization produced by this procedure has the structure of an algebraic weak factorization system; in particular, the left factor is a free $\C$-coalgebra, therefore lifts against the $\F$-algebras and is hence a trivial cofibration. 

\subsection{The free monad on a pointed endofunctor}\label{freemonad}

We now explain what precisely we mean by ``free monad'' and state Kelly's abstract existence result. In the next section, we then verify that the functor $R$ satisfies his conditions under certain set-theoretical assumptions on the underlying category $\cC$. 

Let $R$ be a pointed endofunctor on a category $\cC$. The \emph{algebraically free} monad on $R$ is a monad $\F$ together with an isomorphism $\mathbf{Alg}_\F \cong \mathbf{Alg}_R$ over $\cC$. When $\cC$ is locally small and complete, algebraically free monads coincide with so-called \emph{free monads}, which are defined in \cite[22.2-4]{kellyunified}. We use the terminology ``free monad'' because it is shorter.

Furthermore,  under good conditions, there is a canonical construction that produces the free monad on $R$, in which case we say the free monad \emph{exists constructively}. The construction is via a colimit defined using transfinite induction; the ``good conditions'' guarantee that this construction \emph{converges}.\footnote{Compare with Quillen's small object argument, which never converges, but must be terminated artificially.} 

\begin{rmk}\label{rmk:wellpointed}
A na\"{i}ve approach might be to try and define $F$ to be the colimit of \[ \id \to R \to R^2 \to R^3 \to \cdots \] This works in the case where $R$ is \emph{well-pointed}, meaning $\eta R = R \eta \co R \to R^2$, but not otherwise. Interestingly, the failure of Cole's functor $R$ to be well-pointed precisely highlights the subtle point at which his argument breaks down. We'll say more about this in the appendix, cf.~Section \ref{app2}. 
\end{rmk}

The correct construction is due to Kelly, the first few stages of which we will describe explicitly in the appendix. We make use of only a special case of his theorem  \cite[22.3]{kellyunified}.

In order to state it, we need to introduce a little bit of terminology. An \emph{orthogonal factorization system} $(\cE,\cM)$ on a category $\cC$ is a weak factorization system for which both the factorizations and the liftings are unique. It is called \emph{well-copowered} if every object in $\cC$ has a mere set of $\cE$-quotients, up to isomorphism. When $\cC$ is cocomplete it follows that the maps in $\cE$ are epimorphisms \cite[1.3]{kellyunified}. 
\begin{rmk}
Note that any category that is cocomplete and so that each object has only a sets worth of epimorphism-quotients --- 
a condition satisfied by all categories one meets in practice --- has a functorial factorization where the left factor is an epimorphism and the right factor is a strong monomorphism, see \cite[4.4.3]{borceuxhandbookI}. The dual hypotheses are equally common in our setting. In practice this means that there are always at least two choices for $(\cE,\cM)$: (epimorphisms, strong monomorphisms) and (strong epimorphisms, monomorphisms).
\end{rmk}

A cocone in $\cC$ all of whose legs are elements of $\cM$ is called an $\cM$-\emph{cocone} or an $\cM$-\emph{colimit} in the case it is a colimit cocone. It follows from the right cancelation property of $\cM$ that the morphisms in the diagram also lie in $\cM$, but our condition is stronger. In what follows, we will implicitly identify a regular cardinal $\alpha$ with its initial ordinal, such that $\alpha$ indexes a (transfinite) sequence whose objects are $\beta < \alpha$.

We are now ready to state Kelly's theorem.

\begin{thm}[Kelly]\label{kellythm} Suppose $\cC$ is complete, cocomplete, and locally small. If a pointed endofunctor $R$ on $\cC$ satisfies the following ``smallness'' condition:
\begin{itemize}
\item[($\dagger$)] there is a well-copowered orthogonal factorization system $(\cE,\cM)$ on $\cC$ and a regular cardinal $\alpha$ so that $R$ sends $\alpha$-indexed $\cM$-colimits  to colimits,
\end{itemize}
then the free monad $\F$ on $R$ exists constructively. 
\end{thm}

If $R$ is a pointed endofunctor on $\cC^{\bf 2}$ over $\cod$, then each functor and natural transformation in the free monad construction is constant on its codomain component. It follows that $\F$ is a monad over cod and hence gives rise to a functorial factorization. Furthermore, this observation allows us to weaken the smallness condition for such $R$: It suffices to show that $R$ preserves $\cM$-colimits of the form 
\begin{equation}\label{eq:Mcolimit} \xymatrix{ X_0 \ar[drr]_{f_0} \ar[r]^{m_0} & X_1 \ar[dr]^{f_1} \ar[r]^{m_1} & \cdots \ar[r]^{m_{\beta-1}} & X_\beta \ar[r]^{m_\beta} \ar[dl]_{f_\beta} & \cdots \ar[r] & X_\alpha \ar[dlll]^*+{\labelstyle\mathrm{colim}_{\beta < \alpha}\, f_\beta = f_\alpha} \\ & & Y}\end{equation} 
See \cite[p.~31]{garnercofibrantly}.

\subsection{Smallness}\label{smallness}

The functor $R$ of \eqref{eq:keydiagram} is constructed by means of various topologically enriched limits and colimits in $\cC$. In this section, we will show that if $\cC$ satisfies a set-theoretical condition, then Cole's functor $R$ satisfies the necessary smallness condition to guarantee convergence of the free monad sequence. This condition is very similar to Cole's ``cofibration hypothesis'' \cite[4.1]{colemany}. Indeed, as explained there, work of Lewis \cite{lewisthesis} shows that many topologically bicomplete categories of interest satisfy our condition.

\begin{defn}\label{defn:monomorphism} Suppose $(\cE,\cM)$ is a well-copowered orthogonal factorization system on a topologically bicomplete category $\cC$. We say $\cC$ satisfies the \emph{monomorphism hypothesis} if there is some regular cardinal $\alpha$ so that the mapping path space functor $N \co \cC^{\bf 2} \to \cC$ preserves $\cM$-colimits of diagrams of the form \eqref{eq:Mcolimit}, in the sense that the natural map \[ \mathrm{colim}_{\beta < \alpha}\, Nf_\beta \to N(\mathrm{colim}_{\beta <\alpha}\, f_\beta) = Nf_\alpha\] is an isomorphism.
\end{defn}

\begin{lemma}\label{smallnesslemma}
If $\cC$ is a topologically bicomplete category satisfying the monomorphism hypothesis, then the functor $R$ constructed in \ref{colesconstruction} satisfies condition $(\dagger)$ of Theorem \ref{kellythm}. Therefore, the free monad $\F$ on $R$ exists constructively. 
\end{lemma}
\begin{proof}
By the remarks made at the end of \S\ref{freemonad}, we need to check that $R$ sends an $\cM$-colimit diagram of the form \eqref{eq:Mcolimit} to a colimit cone in $\cC^{\mathbf{2}}$. Clearly, it is enough to verify this for $\dom \circ R: f \mapsto Ef$.  By the monomorphism hypothesis and the fact that $-\otimes I$ is a left adjoint, the colimits of the top and left corners of the following diagram 
\[\xymatrix@!0{ & &   & &  & \\ & Nf_1  \ar[rrr] \ar'[d][dd] \ar@{..>}[ur]  & & & X_1 \ar[dd] \ar@{..>}[ur]  \\ Nf_0 \ar[dd] \ar[rrr] \ar[ur] & & &  X_0 \ar[ur]\ar[dd] & &\\ & Nf_1 \otimes I \ar[rrr] \ar@{..>}[ur] & & &  Ef_1 \po \ar@{..>}[ur] \\ Nf_0 \otimes I \ar[ur]  \ar[rrr] & & & Ef_0 \po \ar[ur] }\] 
are $Nf_\alpha$, $X_\alpha$, and $Nf_\alpha \otimes I$. The pushout of these objects is, by definition, $Ef_\alpha$. Because colimits commute with colimits, the canonical map 
\[ \mathrm{colim}_{\beta < \alpha}\, Ef_\beta \to E(\mathrm{colim}_{\beta <\alpha}\, f_\beta) = Ef_\alpha\] 
is thus an isomorphism, and we conclude that $R$ preserves $\cM$-colimits, as desired.
\end{proof}

\begin{example} The category $\Top$ satisfies the monomorphism hypothesis for the orthogonal factorization system in which $\cE$ is the surjections and $\cM$ is the subspace inclusions. This is a consequence of an observation made by Lewis \cite{lewisthesis}, summarized in \cite[\S 4]{colemany}, about pullbacks of countable sequential $\cM$-colimits. The orthogonal factorization system $(\cE,\cM)$ lifts to $\Top_*$. Because pullbacks, sequential colimits, and mapping path objects coincide with those of spaces, this category also satisfies the monomorphism hypothesis. Similarly, $(\cE,\cM)$ lifts to $G$-spaces, where $G$ is a topologically group, or indeed to any space-valued diagram category. Limits and colimits in such categories are computed pointwise, so again the monomorphism hypothesis is satisfied. 
\end{example}

\begin{example} 
Other interesting examples are given by various categories of topological spectra. To illustrate how the condition of Definition \ref{defn:monomorphism} might be checked, we include a sketch in the case of diagram spectra \cite{MMSS}. 

To this end, let $\cD$ be a small based topological category. Let $R$ be a monoid in the closed symmetric monoidal category of continuous functors from $\cD$ to $\Top_*$.  Following \cite[1.10]{MMSS}, the category of $\cD$-spectra over $R$ is isomorphic to the category of $R$-modules. The orthogonal factorization system $(\cE,\cM)$ on $\Top_*$ defines an orthogonal factorization system on the functor category $\Top_*^{\cD}$ with the classes and factorizations defined pointwise. In any category with an orthogonal factorization system $(\cE,\cM)$ and monad $T$, if $T$ preserves the class $\cE$ then the orthogonal factorization system lifts to the category of $T$-algebras \cite[2.3.7]{cockett}. 

In our situation, the monad $R \wedge -$ is defined via a colimit, which is ultimately computed pointwise in $\Top_*$.   Because smash products preserve surjections in $\Top_*$ and the left class of any orthogonal factorization system is stable under colimits in the arrow category,  this orthogonal factorization system lifts to the category of $R$-modules. The forgetful functor from $R$-modules to 
$\Top_*^{\cD}$ preserves both limits and colimits; hence this category satisfies the monomorphism hypothesis. 
\end{example}

\begin{rmk}
We should remark that any locally presentable topologically bicomplete category $\cC$ also satisfies our hypothesis. Even though local presentability doesn't seem to be a reasonable assumption in the topological context, it might be in adaptations of our methods to other situations, e.g., categories enriched, tensored, and cotensored in appropriate categories other than $\Top$. Furthermore, this observation highlights the set-theoretical nature of the monomorphism hypothesis.  To prove the claim, note that in a locally $\alpha$-presentable category $\cC$ there exists a set of generators $g$ so that the associated representable functors $\cC(g,-)$ detect isomorphisms and preserve $\alpha$-filtered colimits. This, together with the fact that filtered colimits and  finite limits commute in $\Set$, can be used to prove that any locally $\alpha$-presentable category satisfies the monomorphism hypothesis for the (strong epimorphism, monomorphism) orthogonal factorization system and for the cardinal $\alpha$. We leave the remaining details to the reader.  
\end{rmk}

\subsection{The main theorem}

Our main result is the following theorem, which asserts that for a very general class of topologically bicomplete categories, applying the free monad construction to Cole's step-one right functor $R$ yields an algebraic weak factorization system.

\begin{thm}\label{technicalmainthm} If $\cC$ is topologically bicomplete category satisfying the monomorphism hypothesis, the functor $R$ of (\ref{eq:keydiagram}) satisfies the condition ($\dagger$) and furthermore the functorial factorization $(C,F)$ constructed by the free monad sequence is an algebraic weak factorization system.
\end{thm}

In particular, the right factor $Ff$ is a (free) $\F$-algebra and hence an $R$-algebra, and hence a fibration. The left factor $Cf$ is a (free) $\C$-coalgebra, and in particular lifts against all $\F$-algebras by Lemma \ref{lem:lifting}. It follows that $Cf$ is a trivial cofibration. Thus, it follows that: 

\begin{thm}\label{mainthm}
On any  topologically bicomplete category $\cC$ satisfying the monomorphism hypothesis there exists an algebraic weak factorization system $(\C, \F)$ whose right class consists precisely of the $h$-fibrations, while the left class is  the class of strong $h$-cofibrations that are homotopy equivalences.  
\end{thm}

By work of \cite{schwaenzlvogt}, nicely summarized in \cite[4.3.1 and 4.3.3]{maysig},  we have an immediate corollary:
\begin{cor}\label{maincor}
Any topologically bicomplete category $\cC$ satisfying the monomorphism hypothesis admits an $h$-model structure.
\end{cor}

\begin{proof}[Proof of Theorem \ref{technicalmainthm}]
Using Lemma \ref{smallnesslemma}, Cole's step-one right functor $R$ satisfies the smallness condition $(\dagger)$ required to construct the free monad $\F$. By Propositions \ref{charfib} and \ref{prop:boxslashcomp}, $\mathbf{Alg}_{\F} \cong \mathbf{Alg}_R \cong \cI^\boxslash$ admits a vertical composition law.  We conclude that the resulting free monad is part of an algebraic weak factorization system by applying Theorem \ref{thm:recognition}.
\end{proof}

\begin{rmk}\label{secondproofrmk}
An alternative proof of the main theorem avoids Theorem \ref{thm:recognition}. Instead, one can show that the pointed endofunctor $L$ carries a natural comonad structure, where the comultiplication $\delta$ is constructed as follows: First note that, by Lemma \ref{pathchar}, the commutative square
\[
\xymatrix{Nf \ar[r]^-{\phi_f} \ar[d]_{i_0} & X \ar[d]^{Lf} \\
Nf \otimes I \ar[r]_-{\psi_f} & Ef
}
\]
is classified by a map $\tilde{\delta}_f \co Nf \to N{Lf}$. Pushing out, we obtain a map $\delta_f \co Ef \to ELf$. This defines (the codomain component of) a natural transformation $\vec{\delta} \co L \to L^2$ making $L$ into a comonad.

Garner shows \cite[4.21-22]{garnerunderstanding} that the free monad can be constructed in the category of functorial factorizations whose left factors is a comonad. Hence, this extra structure is enough to guarantee the existence of an appropriate algebraic weak factorization system, thereby providing another proof of Theorem \ref{technicalmainthm}.  
\end{rmk}

These arguments emphasize the versatility of the theory of algebraic weak factorization systems. While Garner shows that \emph{cofibrantly generated} algebraic weak factorization systems are produced by Kelly's free monad construction, we demonstrate that these techniques also work in non-cofibrantly generated cases.

\section{Appendix}\label{app}

\subsection{Cole's construction, explicitly}\label{app1}
In the light of Remark \ref{rmk:wellpointed}, Cole's construction does not produce the free monad on the pointed endofunctor $R$ because $R$ is not \emph{well-pointed}.  In this appendix we explain in detail why the maps $\psi_{R^nf}$ don't glue to give a lift in the diagram
\begin{equation}\label{nonlift}
\xymatrix{ NR^\omega f \ar[r]^-{\phi_{R^{\omega}f}} \ar[d]_{i_0} & Z \ar[d]^{R^\omega f} \\
NR^\omega f \otimes I \ar[r]_-{\widehat{\chi_{R^{\omega}f}}} \ar@{-->}[ru] & Y
}
\end{equation}
the existence of which is equivalent to $R^\omega f$ being a fibration by Remark \ref{redundantrmk}. To this end, we will explicitly describe the underlying sets of the objects of the first two iterations of Cole's construction \cite[Ch.~5]{colemany} applied in the case of topological spaces. 

Let $f\co X \to Y$ be a morphism of topological spaces. The points of $Nf$ are pairs $(x,p)$ with $x \in X$ and $p\co I \to Y$ a path in $Y$ starting at $f(x)$. The map $\phi_f\co Nf \to X$ simply forgets the path, i.e., $(x,p) \mapsto x$, from which we deduce that, as a set, the pushout\footnote{denoted by $Z_1$ in \cite{colemany}} $Ef$ consists of two kinds of elements:
\begin{enumerate}
 \item[(i)] $(x, c_{f(x)}, 0) \in Ef$, with $x \in X$ and $c_{f(x)}$ being the constant path at $f(x)$
 \item[(ii)] $(x,p,t) \in Ef$ with $(p,x) \in Nf$, $t \in (0,1]$.
\end{enumerate}
The induced map $Rf\co Ef \to Y$ sends an element $(x,p,t) \in Ef$ to the point $p(t)$. 

Similarly, we can describe the space $E{Rf}$ as a set. Points in $NRf$ are pairs consisting of $(x,p,t) \in Ef$ together with a path $p' \co I \to Y$ starting at $p(t)$. Points in $ERf$  are of the general form $(x, p, t, p', t')$ with $x \in X$, $t,t' \in I$, $p\co I \to Y$ a path starting at $f(x)$, and $p'\co I \to Y$ a path starting at $p(t)$. There are four types:
\begin{enumerate}
 \item $(x, c_{f(x)}, 0, c_{f(x)}, 0)$
 \item $(x, p, t, c_{p(t)}, 0)$ with $t\in (0,1]$
 \item $(x, c_{f(x)}, 0, p', t')$ with $t' \in (0,1]$
 \item $(x, p, t, p', t')$, with $t, t' \in (0,1]$
\end{enumerate}
Recall that the object $Z$ in Cole's proposed factorization $X \xrightarrow{j} Z \xrightarrow{R^\omega f} Y$ is defined to be the colimit of the $Z_{n+1} = ER^nf$ with respect to the maps $LR^nf$. Since all the $LR^nf$ are closed embeddings, in order for the maps $\psi_{R^{n}f}$ to glue, the following square has to commute
\begin{equation}\label{noncomm}
\xymatrix{N{R^nf} \otimes I \ar[r]^-{\psi_{R^nf}} \ar[d]_{N(LR^{n}f,\id) \otimes I} & E{R^nf} \ar[d]^{LR^{n+1}f} \\
N{R^{n+1}f} \otimes I \ar[r]_-{\psi_{R^{n+1}f}} & E{R^{n+1}f}
}
\end{equation}
the left vertical map being the one with respect to which the colimit object $NR^\omega f$ is formed. Observe that, by construction, this square commutes if the right vertical map $LR^{n+1}f$ is replaced by $E(LR^{n}f,\id)$. Using the notation $(R,\vec{\eta})$ for the pointed endofunctor, the map $LR^{n+1}f$ is the domain component of $\vec{\eta}_{R^{n+1}f}$ while $E(LR^nf,\id)$ is the domain component of $R\vec{\eta}_{R^nf}$.

But specializing to $n=0$, we see that the $LRf$ sends points in $Ef$ of type (ii) to points of type (2) in $E{Rf}$, while $E(Lf,\id)$ maps those points to elements of type (3). Therefore, the diagram \eqref{noncomm} does not commute and the maps $\psi_{R^nf}$ do not glue to give a lift in \eqref{nonlift}. \\

\begin{rmk}
We should note that this argument can't rule out the possibility that the map $R^{\omega}f$ is a fibration; however, this seems very difficult to check, as the crafted candidate fails to provide a lift in the colimit. 
\end{rmk}

\subsection{Our construction, explicitly}\label{app2}
The free monad construction is described in \cite[4.16]{garnerunderstanding}. For the reader's convenience, we describe the first few stages of the construction of the functorial factorization whose right factor is the free monad on the pointed endofunctor of $(L,R)$.

Let $C^1 = L$ and let $F^1 = R$. Define $C^2$ and $F^2$ using the coequalizer $\omega^2_f$
\[ \xymatrix{ X \ar@{=}[r] \ar[d]_{Lf} & X \ar[d]^{LRf \cdot Lf} \ar@{=}[r] & X \ar@{-->}[d]^{C^2f} \\ Ef \ar@<.5ex>[r]^{LRf} \ar@<-.5ex>[r]_{E(Lf,\id)} \ar[d]_{Rf} & ERf \ar@{-->}[r]_{\omega^2_f} \ar[d]^{R^2f} & E^2f \ar@{-->}[d]^{F^2f} \\ Y \ar@{=}[r] & Y \ar@{=}[r] & Y}\] Specializing to the category of topological spaces, in the notation of the previous section, the quotient map $\omega^2_f$ identifies points of types (2) and (3) in the obvious way.

Continuing, define $C^3$ and $F^3$ using the coequalizer $\omega^3_f$
\[\xymatrix{ X \ar@{=}[r] \ar[d]_{LRf \cdot Lf} & X \ar[d]_{LF^2f \cdot C^2f} \ar@{=}[r] & X \ar@{-->}[d]^{C^3f} \\ ERf \ar@<.5ex>[r] \ar@<-.5ex>[r] \ar[d]_{R^2f} & EF^2f \ar[d]_{RF^2f} \ar@{-->}[r]^{\omega^3_f} & E^3f \ar@{-->}[d]^{F^3f} \\ Y \ar@{=}[r] & Y \ar@{=}[r] & Y}\]
 of the following parallel pair of morphisms 
\[ \xymatrix{X \ar[d]|{LRf \cdot Lf} \ar@{=}[r] & X \ar[d]|{C^2f} \ar@{=}[r] & X \ar[d]|{LF^2f \cdot C^2f} \\ ERf \ar[d]_{R^2f} \ar[r]^{\omega^2_f} & E^2f \ar[d]_{F^2f} \ar[r]^{LF^2f} & EF^2f \ar[d]^{RF^2f} \\ Y \ar@{=}[r] & Y \ar@{=}[r] & Y}
\xymatrix{ X \ar@{=}[r] \ar[d]|{LRf \cdot Lf} & X \ar[d]|{LR^2f \cdot LRf \cdot Lf} \ar@{=}[r] & X \ar[d]|{LF^2f \cdot C^2f} \\ ERf \ar[r]_{E(LR,\id)} \ar[d]_{R^2f} & ER^3f \ar[d]_{R^3f} \ar[r]_{E(\omega^2f,\id)} & EF^2f \ar[d]^{RF^2f} \\ Y \ar@{=}[r] & Y \ar@{=}[r] & Y}
\]
The remaining details are left as an exercise.

\bibliographystyle{alpha}

\end{document}